\documentclass[12pt]{amsart}
\usepackage{amsmath, mathrsfs, mathalfa, rsfso}
\usepackage{amsxtra}
\usepackage{amscd}
\usepackage{amsthm}
\usepackage{amsfonts}
\usepackage{amssymb}
\usepackage{eucal}
\usepackage{verbatim}
\usepackage[all]{xy}
\usepackage{color}
\definecolor{darkspringgreen}{rgb}{0.09, 0.45, 0.27}
\usepackage{stmaryrd}
\usepackage{geometry}
\usepackage[utf8]{inputenc}
\usepackage[T1]{fontenc}
\usepackage[english]{babel}
\usepackage{microtype} 
\usepackage{tabularx} 
\usepackage{booktabs} 
\usepackage{amsmath, amsfonts,amssymb,amsthm}
\usepackage[all]{xy}
\usepackage{tikz}
\usepackage{appendix}
\usepackage{etoolbox}%
\usepackage[T1]{fontenc}
\usepackage{lmodern}
\usepackage{enumitem}
\usepackage{color}
\definecolor{darkspringgreen}{rgb}{0.09, 0.45, 0.27}
\usepackage{stmaryrd}
\usepackage[pdftex,bookmarks=false,colorlinks=true,debug=true,citecolor=darkspringgreen, naturalnames=true,pdfnewwindow=true]{hyperref}
\patchcmd{\footnotemark}{\stepcounter{footnote}}{\refstepcounter{footnote}}{}{}
\newtheorem{thm}[subsubsection]{Theorem}
\newtheorem{cor}[subsubsection]{Corollary}

\newtheorem{lem}[subsubsection]{Lemma}
\newtheorem{prop}[subsubsection]{Proposition}
\theoremstyle{definition}
\newtheorem{definition}[subsubsection]{Definition}
\newtheorem{cond}[subsubsection]{Condition}
\usepackage[normalem]{ulem}

\newcommand{\stkout}[1]{\ifmmode\text{\sout{\ensuremath{#1}}}\else\sout{#1}\fi}
\usepackage{xcolor,cancel}

\theoremstyle{remark}
\newtheorem{rem}[subsubsection]{Remark}

\newcommand{\nc}{\newcommand}
\nc{\renc}{\renewcommand} \nc{\ssec}{\subsection}
\nc{\sssec}{\subsubsection}

\nc{\on}{\operatorname} \nc{\wh}{\widehat}
\nc\ol{\overline} \nc\ul{\underline} \nc\wt{\widetilde}

\nc{\BA}{{\mathbb{A}}} \nc{\BC}{{\mathbb{C}}} \nc{\BF}{{\mathbb{F}}}
\nc{\BD}{{\mathbb{D}}} \nc{\BG}{{\mathbb{G}}} \nc{\BQ}{{\mathbb{Q}}}
\nc{\BM}{{\mathbb{M}}} \nc{\BN}{{\mathbb{N}}} \nc{\BO}{{\mathbb{\bfO}}}
\nc{\BP}{{\mathbb{P}}} \nc{\BR}{{\mathbb{R}}}

\nc{\CA}{{\mathcal{A}}} \nc{\CB}{{\mathcal{B}}} \nc{\CalD}{{\mathcal{D}}}
\nc{\CE}{{\mathcal{E}}} \nc{\CF}{{\mathcal{F}}}
\nc{\CG}{{\mathcal{G}}} \nc{\CH}{{\mathcal{H}}}
\nc{\CI}{{\mathcal{I}}} \nc{\CK}{{\mathcal{K}}} \nc{\CL}{{\mathcal{L}}}
\nc{\CM}{{\mathcal{M}}} \nc{\CN}{{\mathcal{N}}}
\nc{\CO}{{\mathcal{\bfO}}} \nc{\CP}{{\mathcal{P}}}
\nc{\CQ}{{\mathcal{Q}}} \nc{\CR}{{\mathcal{R}}}
\nc{\CS}{{\mathcal{S}}} \nc{\CT}{{\mathcal{T}}}
\nc{\CU}{{\mathcal{U}}} \nc{\CV}{{\mathcal{V}}}  \nc{\CY}{{\mathcal Y}}
\nc{\CW}{{\mathcal{W}}} \nc{\CZ}{{\mathcal{Z}}}

\nc{\cM}{{\check{\mathcal M}}{}} \nc{\csM}{{\check{\mathcal A}}{}}
\nc{\oM}{{\overset{\circ}{\mathcal M}}{}}
\nc{\obM}{{\overset{\circ}{\mathbf M}}{}}
\nc{\oCA}{{\overset{\circ}{\mathcal A}}{}}
\nc{\obA}{{\overset{\circ}{\mathbf A}}{}}
\nc{\ooM}{{\overset{\circ}{M}}{}}
\nc{\osM}{{\overset{\circ}{\mathsf M}}{}}
\nc{\vM}{{\overset{\bullet}{\mathcal M}}{}}
\nc{\nM}{{\underset{\bullet}{\mathcal M}}{}}
\nc{\oD}{{\overset{\circ}{\mathcal D}}{}}
\nc{\obD}{{\overset{\circ}{\mathbf D}}{}}
\nc{\oA}{{\overset{\circ}{\mathbb A}}{}}
\nc{\op}{{\overset{\bullet}{\mathbf p}}{}}
\nc{\cp}{{\overset{\circ}{\mathbf p}}{}}
\nc{\oU}{{\overset{\bullet}{\mathcal U}}{}}
\nc{\ofZ}{{\overset{\circ}{\mathfrak Z}}{}}

\nc{\ff}{{\mathfrak{f}}} \nc{\fv}{{\mathfrak{v}}}
\nc{\fa}{{\mathfrak{a}}} \nc{\fb}{{\mathfrak{b}}}
\nc{\fd}{{\mathfrak{d}}} \nc{\fe}{{\mathfrak{e}}}
\nc{\fg}{{\mathfrak{g}}} \nc{\fgl}{{\mathfrak{gl}}}
\nc{\fh}{{\mathfrak{h}}} \nc{\fri}{{\mathfrak{i}}}
\nc{\fj}{{\mathfrak{j}}} \nc{\fk}{{\mathfrak{k}}} \nc{\fl}{{\mathfrak{l}}}
\nc{\fm}{{\mathfrak{m}}} \nc{\fn}{{\mathfrak{n}}}
\nc{\ft}{{\mathfrak{t}}} \nc{\fu}{{\mathfrak{u}}}
\nc{\fw}{{\mathfrak{w}}} \nc{\fz}{{\mathfrak{z}}}
\nc{\fp}{{\mathfrak{p}}} \nc{\fq}{{\mathfrak{q}}} \nc{\frr}{{\mathfrak{r}}}
\nc{\fs}{{\mathfrak{s}}} \nc{\fsl}{{\mathfrak{sl}}}
\nc{\hsl}{{\widehat{\mathfrak{sl}}}}
\nc{\hgl}{{\widehat{\mathfrak{gl}}}}
\nc{\hg}{{\widehat{\mathfrak{g}}}}
\nc{\chg}{{\widehat{\mathfrak{g}}}{}^\vee}
\nc{\hn}{{\widehat{\mathfrak{n}}}}
\nc{\chn}{{\widehat{\mathfrak{n}}}{}^\vee}

\nc{\fA}{{\mathfrak{A}}} \nc{\fB}{{\mathfrak{B}}} \nc{\fC}{{\mathfrak{C}}}
\nc{\fD}{{\mathfrak{D}}} \nc{\fE}{{\mathfrak{E}}}
\nc{\fF}{{\mathfrak{F}}} \nc{\fG}{{\mathfrak{G}}} \nc{\fH}{{\mathfrak{H}}}
\nc{\fI}{{\mathfrak{I}}} \nc{\fJ}{{\mathfrak{J}}}
\nc{\fK}{{\mathfrak{K}}} \nc{\fL}{{\mathfrak{L}}}
\nc{\fM}{{\mathfrak{M}}} \nc{\fN}{{\mathfrak{N}}}
\nc{\frP}{{\mathfrak{P}}} \nc{\fQ}{{\mathfrak{Q}}}
\nc{\fS}{{\mathfrak{S}}} \nc{\fT}{{\mathfrak{T}}} \nc{\fU}{{\mathfrak{U}}}
\nc{\fV}{{\mathfrak{V}}} \nc{\fW}{{\mathfrak{W}}}
\nc{\fX}{{\mathfrak{X}}} \nc{\fY}{{\mathfrak{Y}}}
\nc{\fZ}{{\mathfrak{Z}}}

\nc{\ba}{{\mathbf{a}}}
\nc{\bb}{{\mathbf{b}}} \nc{\bc}{{\mathbf{c}}}
\nc{\be}{{\mathbf{e}}} \nc{\bj}{{\mathbf{j}}} \nc{\bm}{{\mathbf{m}}}
\nc{\bn}{{\mathbf{n}}} \nc{\bp}{{\mathbf{p}}}
\nc{\bq}{{\mathbf{q}}} \nc{\br}{{\mathbf{r}}} \nc{\bt}{{\mathbf{t}}}
\nc{\bfu}{{\mathbf{u}}} \nc{\bv}{{\mathbf{v}}}
\nc{\bx}{{\mathbf{x}}} \nc{\by}{{\mathbf{y}}} \nc{\bz}{{\mathbf{z}}}
\nc{\bw}{{\mathbf{w}}} \nc{\bA}{{\mathbf{A}}}
\nc{\bB}{{\mathbf{B}}} \nc{\bC}{{\mathbf{C}}}
\nc{\bD}{{\mathbf{D}}} \nc{\bF}{{\mathbf{F}}} \nc{\bG}{{\mathbf{G}}}
\nc{\bH}{{\mathbf{H}}} \nc{\bI}{{\mathbf{I}}} \nc{\bJ}{{\mathbf{J}}}
\nc{\bK}{{\mathbf{K}}} \nc{\bM}{{\mathbf{M}}} \nc{\bN}{{\mathbf{N}}}
\nc{\bO}{{\mathbf{\bfO}}} \nc{\bS}{{\mathbf{S}}} \nc{\bT}{{\mathbf{T}}}
\nc{\bU}{{\mathbf{U}}} \nc{\bV}{{\mathbf{V}}} \nc{\bW}{{\mathbf{W}}}
\nc{\bX}{{\mathbf{X}}}
\nc{\bY}{{\mathbf{Y}}} \nc{\bP}{{\mathbf{P}}}
\nc{\bZ}{{\mathbf{Z}}} \nc{\bh}{{\mathbf{h}}}

\nc{\sA}{{\mathsf{A}}} \nc{\sB}{{\mathsf{B}}}
\nc{\sC}{{\mathsf{C}}} \nc{\sD}{{\mathsf{D}}}
\nc{\sE}{{\mathsf{E}}} \nc{\sF}{{\mathsf{F}}} \nc{\sG}{{\mathsf{G}}}
\nc{\sI}{{\mathsf{I}}} \nc{\sK}{{\mathsf{K}}} \nc{\sL}{{\mathsf{L}}}
\nc{\sfm}{{\mathsf{m}}} \nc{\sM}{{\mathsf{M}}} \nc{\sO}{{\mathsf{\bfO}}}
\nc{\sQ}{{\mathsf{Q}}} \nc{\sP}{{\mathsf{P}}}
\nc{\sT}{{\mathsf{T}}} \nc{\sZ}{{\mathsf{Z}}}
\nc{\sV}{{\mathsf{V}}} \nc{\sW}{{\mathsf{W}}}
\nc{\sfp}{{\mathsf{p}}} \nc{\sq}{{\mathsf{q}}} \nc{\sr}{{\mathsf{r}}}
\nc{\st}{{\mathsf{t}}} \nc{\sfb}{{\mathsf{b}}}
\nc{\sfc}{{\mathsf{c}}} \nc{\sd}{{\mathsf{d}}}
\nc{\sz}{{\mathsf{z}}}

\nc{\tA}{{\widetilde{\mathbf{A}}}}
\nc{\tB}{{\widetilde{\mathcal{B}}}}
\nc{\tg}{{\widetilde{\mathfrak{g}}}} \nc{\tG}{{\widetilde{G}}}
\nc{\TM}{{\widetilde{\mathbb{M}}}{}}
\nc{\tO}{{\widetilde{\mathsf{\bfO}}}{}}
\nc{\tU}{{\widetilde{\mathfrak{U}}}{}} \nc{\TZ}{{\tilde{Z}}}
\nc{\tx}{{\tilde{x}}} \nc{\tbv}{{\tilde{\bv}}}
\nc{\tfP}{{\widetilde{\mathfrak{P}}}{}} \nc{\tz}{{\tilde{\zeta}}}
\nc{\tmu}{{\tilde{\mu}}}

\nc{\urho}{\underline{\rho}} \nc{\uB}{\underline{B}}
\nc{\uC}{{\underline{\mathbb{C}}}} \nc{\ui}{\underline{i}}
\nc{\uj}{\underline{j}} \nc{\ofP}{{\overline{\mathfrak{P}}}}
\nc{\oB}{{\overline{\mathcal{B}}}}
\nc{\og}{{\overline{\mathfrak{g}}}} \nc{\oI}{{\overline{I}}}

\nc{\eps}{\varepsilon} \nc{\hrho}{{\hat{\rho}}}
\nc{\blambda}{{\boldsymbol{\lambda}}} \nc{\bmu}{{\boldsymbol{\mu}}} \nc{\bnu}{{\boldsymbol{\nu}}}

\nc{\one}{{\mathbf{1}}} \nc{\two}{{\mathbf{t}}}

\nc{\Sym}{\mathop{\operatorname{\rm Sym}}}
\nc{\Tot}{{\mathop{\operatorname{\rm Tot}}}}
\nc{\Spec}{\mathop{\operatorname{\rm Spec}}}
\nc{\Ker}{{\mathop{\operatorname{\rm Ker}}}}
\nc{\Isom}{{\mathop{\operatorname{\rm Isom}}}}
\nc{\Hilb}{{\mathop{\operatorname{\rm Hilb}}}}
\nc{\deeq}{{\mathop{\operatorname{\rm deeq}}}}
\nc{\End}{{\mathop{\operatorname{\rm End}}}}
\nc{\Ext}{{\mathop{\operatorname{\rm Ext}}}}
\nc{\Hom}{{\mathop{\operatorname{\rm Hom}}}}
\nc{\CHom}{{\mathop{\operatorname{{\mathcal{H}}\it om}}}}
\nc{\GL}{{\mathop{\operatorname{\rm GL}}}}
\nc{\Stab}{{\mathop{\operatorname{\rm Stab}}}}
\nc{\Mir}{{\mathop{\operatorname{\rm Mir}}}}
\nc{\gr}{{\mathop{\operatorname{\rm gr}}}}
\nc{\Id}{{\mathop{\operatorname{\rm Id}}}}
\nc{\perf}{{\mathop{\operatorname{\rm perf}}}}
\nc{\ver}{{\mathop{\operatorname{\rm ver}}}}
\nc{\glob}{{\mathop{\operatorname{\rm glob}}}}
\nc{\ratio}{{\mathop{\operatorname{\rm ratio}}}}
\nc{\Ran}{{\mathop{\operatorname{\rm Ran}}}}
\nc{\Br}{{\mathop{\operatorname{\rm Br}}}}
\nc{\BS}{{\mathop{\operatorname{\rm BS}}}}
\nc{\Conf}{{\mathop{\operatorname{\rm Conf}}}}
\nc{\mix}{{\mathop{\operatorname{\rm mix}}}}
\nc{\Lus}{{\mathop{\operatorname{\rm Lus}}}}
\nc{\KD}{{\mathop{\operatorname{\rm KD}}}}
\nc{\defi}{{\mathop{\operatorname{\rm def}}}}
\nc{\length}{{\mathop{\operatorname{\rm length}}}}
\nc{\supp}{{\mathop{\operatorname{\rm supp}}}}
\nc{\Fun}{{\mathop{\operatorname{\rm Funct}}}}
\nc{\Hei}{{\mathop{\operatorname{\rm Heis}}}}
\nc{\HC}{{\mathcal H}{\mathcal C}}

\nc{\Cliff}{{\mathsf{Cliff}}}
\nc{\loc}{{\operatorname{loc}}}
\nc{\Fl}{{\mathbf{Fl}}} 
\nc{\Ffl}{{\mathcal{F}\ell}}
\nc{\Fib}{{\mathsf{Fib}}}
\nc{\Coh}{{\mathsf{Coh}}} \nc{\FCoh}{{\mathsf{FCoh}}}
\nc{\Perf}{{\mathsf{Perf}}}

\nc{\reg}{{\text{\rm reg}}}
\nc{\gvee}{{\mathfrak g}^{\!\scriptscriptstyle\vee}}
\nc{\tvee}{{\mathfrak t}^{\!\scriptscriptstyle\vee}}
\nc{\nvee}{{\mathfrak n}^{\!\scriptscriptstyle\vee}}
\nc{\bvee}{{\mathfrak b}^{\!\scriptscriptstyle\vee}}
       \nc{\rhovee}{\rh\bfO^{\!\scriptscriptstyle\vee}}

\nc{\cplus}{{\mathbf{C}_+}} \nc{\cminus}{{\mathbf{C}_-}}
\nc{\cthree}{{\mathbf{C}_*}} \nc{\Qbar}{{\bar{Q}}}

\nc{\Gtimes}{\vphantom{j^{X^2}}\smash{\overset{G}{\vphantom{\rule{0pt}{0.30em}}\smash{\times}}}}
\nc{\sGtimes}{\vphantom{j^{X^2}}\smash{\overset{\mathsf G}{\vphantom{\rule{0pt}{0.30em}}\smash{\times}}}}

\nc{\bOmega}{{\overline{\Omega}}}

\nc{\seq}[1]{\stackrel{#1}{\sim}}
\nc{\Char}{{\operatorname{char}}}
\nc{\aff}{{\operatorname{aff}}}
\nc{\fin}{{\operatorname{fin}}}
\nc{\mir}{{\operatorname{mir}}}
\nc{\triv}{{\operatorname{triv}}}
\nc{\ext}{{\operatorname{ext}}}
\nc{\righ}{{\operatorname{right}}}
\nc{\lef}{{\operatorname{left}}}
\nc{\forg}{{\operatorname{forg}}}
\nc{\fid}{{\operatorname{fd}}}
\nc{\modu}{{\operatorname{-mod}}}
\nc{\Gr}{{\mathbf{Gr}}}
\nc{\FT}{{\operatorname{FT}}}
\nc{\Mat}{{\operatorname{Mat}}}
\nc{\MSt}{{\operatorname{MSt}}}
\nc{\sph}{{\operatorname{sph}}}
\nc{\GR}{{\mathbf{Gr}}}
\nc{\Perv}{{\operatorname{Perv}}}
\nc{\Rep}{{\operatorname{Rep}}}
\nc{\Fact}{{\operatorname{FactMod}}}
\nc{\Ind}{{\operatorname{Ind}}}
\nc{\IC}{{\operatorname{IC}}}
\nc{\Bun}{{\operatorname{Bun}}}
\nc{\Proj}{{\operatorname{Proj}}}
\nc{\pt}{{\operatorname{pt}}}
\nc{\bfmu}{{\boldsymbol{\mu}}}
\nc{\bfomega}{{\boldsymbol{\omega}}}
\nc{\calM}{\mathcal M}
\nc{\calA}{\mathcal A}
\nc{\calO}{\mathcal O}
\nc{\cC}{\mathcal C}
\nc{\CC}{\mathbb C}
\nc{\calN}{\mathcal N}
\nc{\grg}{\mathfrak g}
\nc{\tslash}{/\!\!/\!\!/}
\nc\grt{\mathfrak t}
\nc\bfM{\mathbf M}
\nc\bfN{\mathbf N}

\nc\ZZ{\mathbb{Z}}
\nc\calC{\mathcal C}
\nc\calF{\mathcal F}
\nc\calX{\mathcal X}
\nc\calY{\mathcal Y}
\nc\QCoh{\operatorname{QCoh}}
\nc\Shv{\operatorname{Shv}}
\nc\ShvD{\operatorname{D-mod}}
\nc\IndCoh{\operatorname{IndCoh}}
\nc\Maps{\operatorname{Maps}}
\nc\Dmod{D-\operatorname{mod}}
\newcommand\Hecke{\operatorname{Hecke}}
\nc{\calD}{\mathcal D}
\nc\bfO{\mathbf O}
\nc\bfF{\mathbf F}

\nc\GG{\mathbb G}
\nc\calK{\mathcal K}
\nc{\calG}{\mathcal G}
\nc\RHom{\operatorname{RHom}}
\nc\Res{\operatorname{Res}}
\nc\Av{\operatorname{Av}}
\nc\Eis{\operatorname{Eis}}
\nc\oblv{\operatorname{oblv}}
\nc\pr{\operatorname{pr}}
\nc\unit{\operatorname{unit}}
\nc\add{\operatorname{add}}
\nc\ind{\operatorname{ind}}
\nc\coind{\operatorname{coind}}
\nc\sprd{\operatorname{sprd}}
\nc\projection{\operatorname{projection}}
\nc\averaging{\operatorname{averaging}}
\nc\pullback{\operatorname{pullback}}
\nc\grs{\mathfrak s}
\nc{\tilX}{\widetilde X}
\nc\calB{\mathcal B}
\nc\calS{\mathcal S}
\nc\calT{\mathcal T}
\nc\calZ{\mathcal Z}
\nc\LS{\operatorname{LocSys}}
\nc\bfL{\on{\mathbf L}}

\newcommand*\circled[1]
{*{#1}}
\makeatletter
\newcommand{\raisemath}[1]{\mathpalette{\raisem@th{#1}}}
\newcommand{\raisem@th}[3]{\raisebox{#1}{$#2#3$}}
\nc{\binlim}[2][]{\def\@tempa{#1}\@ifnextchar^{\@binlim{#2}}{\@binlim{#2}^{}}}
\def\@binlim#1^#2{\mathbin{\@ifempty{#2}{\mathop{#1}}{\mathop{#1}\@xp\displaylimits\@tempa^{#2}}}}

\makeatother

\nc\cX{{\mathcal X}}

\nc\Gm{{\mathbb G_m}}
\nc{\isoto}
    {\stackrel{\displaystyle\raisemath{-.2ex}{\sim}}\to}
\nc\cD\CalD
\nc\setm\setminus
\nc\cE\CE
\renc\Hecke{\mathit{\CH\kern-.2ex ecke}}
\nc\bsl\backslash
\nc\Fq{\mathbb F_q}
\nc\x\times
\nc\bGO{{\bG_\bO}}
\nc\bla\blambda
\nc\blt\bullet
\nc\opp{{\on{op}}}
\nc\srel\stackrel
\nc\tbx{\binlim{\widetilde\boxtimes{}}}
\nc\into\hookrightarrow
\nc\otto\leftrightarrow
\renc\setminus\smallsetminus
\nc\phitau{\varphi\tau}
\newenvironment{i-ii-iii}{%
\begin{enumerate}
}%
{\end{enumerate}}
\nc\ceil[1]{\lceil#1\rceil}  \nc\floor[1]{\lfloor#1\rfloor}
\nc\Lie{\on{Lie}}
\def\arxiv#1{\href{http://arxiv.org/abs/#1}{\tt arXiv:#1}} \let\arXiv\arxiv
\nc\kap{\kappa}
\nc\gra{\mathfrak a}
\nc\gl{\mathfrak{gl}}
\nc\sTr{\operatorname{sTr}}
\nc\hatG{\widehat{G}}
\nc\calL{\mathcal L}
\nc\Whit{\operatorname{Whit}}
\nc\KM{\operatorname{KM}}
\nc\KL{\operatorname{KL}}

\makeatletter
\renewcommand{\subsection}{\@startsection{subsection}{2}{0pt}{-3ex
plus -1ex minus -0.2ex}{-2mm plus -0pt minus
-2pt}{\normalfont\bfseries}} \makeatother

\numberwithin{equation}{subsection}
\allowdisplaybreaks
\nc\mto{\mapsto }
\nc\en{\enspace }

\numberwithin{equation}{section}
\newtheorem*{rep@theorem}{\rep@title}
\newcommand{\newreptheorem}[2]{%
\newenvironment{rep#1}[1]{%
 \def\rep@title{#2 \ref{##1}}%
 \begin{rep@theorem}}%
 {\end{rep@theorem}}}
\makeatother
\usepackage{tikz-cd}

\usepackage{wasysym, stackengine, makebox, graphicx}

\usepackage{xcolor}
 \newcommand{\ncmd}{\newcommand*}
\newcommand{\rncmd}{\renewcommand*}
\ncmd{\DMO}{\DeclareMathOperator}
\ncmd{\ncmdd}[2]{\ncmd{#1}{{#2}}}

\makeatletter
\ncmd{\DefOps}[1]{\def\OPERATOR@NAME##1{\DeclareMathOperator{##1}{\expandafter\@gobble\string##1}}
    \def\OPERATOR@LIST##1{\ifcat\noexpand\relax\noexpand##1\OPERATOR@NAME##1\expandafter\OPERATOR@LIST\fi}
    \OPERATOR@LIST#1.}

\ncmd{\DefRm}[1]{\def\OPERATOR@NAME##1{\ncmd{##1}{\mathrm{\expandafter\@gobble\string##1}}}
    \def\OPERATOR@LIST##1{\ifcat\noexpand\relax\noexpand##1\OPERATOR@NAME##1\expandafter\OPERATOR@LIST\fi}
    \OPERATOR@LIST#1.}

\ncmd{\phtr}[2]{\lefteqn{#1{\phantom{#2}}}#2}

\def\Alphabet{ABCDEFGHIJKLMNOPQRSTUVWXYZ}
\def\newalph#1#2{\begingroup
    \def\procL@tt@r##1{%
        \@xp\gdef\csname#1\endcsname{#2}}%
    \proc@lph@bet\endgroup}
\def\proc@lph@bet{\@xp\prlist@\Alphabet\relax}
\def\prlist@#1{\ifx#1\relax\else\procL@tt@r{#1}\@xp\prlist@\fi}

\ncmd{\SmSub}[2][]{_\bgroup #2\smsub@{#1}}
\def\smsub@#1{\@ifnextchar_{#1\@smsb}{\egroup}}
\def\@smsb_#1{#1\smsub@{}}

\ncmd{\SmSup}[2][]{^\bgroup #2\smsup@{#1}}
\def\smsup@#1{\@ifnextchar^{#1\@smsp}{\@ifnextchar'{\prime\@xp\smsup@\@xp{\@xp}\@gobble}{\egroup}}}
\def\@smsp^#1{#1\smsup@{}}

\def\@binlim#1_#2{\mathbin{\@ifempty{#2}{\mathop{#1}}{\mathop{#1}\@xp\displaylimits\@tempa_{#2}}}}

\catcode`\=\active
\ncmd{\RedNote}[1]{\Text{%
  \textcolor{red}{\sffamily#1}}}
\ncmd\Text[1]{\ifmmode\text{#1}\else#1\fi}
\ncmd\todo[1][todo]{\RedNote{#1}%
 \ifx\undefined\@todoflag
   \global\let\@todoflag\relax
  \AtEndDocument{\par\vspace{3em}\noindent \RedNote{TO DO:}}%
 \fi
 \edef\pagenmbr{\thepage}%
 \expandafter\redendnote
 \expandafter{\pagenmbr}{#1}}
\ncmd\redendnote[2]{\AtEndDocument{\\$\bullet$\ \RedNote{page #1: #2}}}
\def{\todo}
\ncmd\hidetodos{\rncmd\todo\relax}

\theoremstyle{remark}
\newalph{c#1}{\ensuremath{{\mathcal #1}}}
\newalph{s#1}{{\mathsf #1}}
\newalph{#1#1}{{\mathbb #1}}
\newalph{e#1}{{\EuScript #1}}
\newalph{g#1}{{\mathfrak #1}}
\newalph{r#1}{{\ensuremath{\mathscr #1}}}
   
\ncmd{\Fp}{{\FF_p}}

\DMO{\cHom}{\text{\textrm{\itshape{\cH}\kern-.2ex{}om}}}
\DMO{\cEnd}{\text{\textrm{\itshape{\cE}\kern-.2ex{}nd}}}    \DMO{\cExt}{\text{\textrm{\itshape{\cE}\kern-.2ex{}xt}}}

\ncmd{\sff}{\mathsf f}

\let\Im\undefined   \let\det\undefined
\ncmd{\young}[1]{\vcenter{\begin{Young}#1\crcr\end{Young}}}
\DefOps   { \REnd  \Aut \colim
             \det \tr \rk \curv  \Coker \Im  \ad \Ad }
\DefRm{ \id \Vect \Vac \Pic \Loc \Hitch \Higgs \Eu \Fr}

\ncmd{\RGam}{\text{\upshape R}\Gamma}
\DMO{\chr}{char}
\ncmd{\angs}[1]{\langle#1\rangle}
\ncmd{\hmod}{\text{\upshape-mod}}
\ncmd{\cxym}[1]{\ensuremath{\vcenter{\xymatrix{#1}}}}
\def\arxiv#1{\href{http://arxiv.org/abs/#1}{\tt arXiv:#1}} \let\arXiv\arxiv
\makeatother

\title[Iwahori Kontsevich compactification]{A resolution of singularities of Drinfeld compactification with an Iwahori structure}
\date{}
\author{Ruotao Yang}

\begin{document}
\maketitle

\begin{abstract}
The Drinfeld compactification $\overline{\Bun}{}_B'$ of the moduli stack $\Bun_B'$ of Borel bundles on a curve $X$ with an Iwahori structure is important in the geometric Langlands program. It is closely related to the study of representation theory. In this paper, we construct a resolution of singularities of it using a modification of Justin Campbell's construction of the Kontsevich compactification. Furthermore, the moduli stack ${\Bun}_B'$ admits a stratification indexed by the Weyl group. For each stratum, we construct a resolution of singularities of its closure. Then we use this resolution of singularities to prove a universally local acyclicity property, which is useful in the quantum local Langlands program.
\end{abstract}
\setcounter{tocdepth}{2}
\section{Introduction}
\subsection{Notations}
Let $G$ be a connected reductive group defined over an algebraically closed field $\mathsf{k}$. Let $B$ be a Borel subgroup, and $T$ be the maximal torus of $B$. We denote by $W$ the Weyl group of $G$. 

For a projective, smooth, and connected curve $X$ of genus $g$ over $\mathsf{k}$, we denote by $\Bun_B$ the algebraic stack classifying  $B$-bundles on $X$. Fix a point $x\in X$, we denote by $\Bun_B'$ the algebraic stack classifying the data $(\mathscr{P}_B, \epsilon_{{x}})$, where $\mathscr{P}_B$ is a $B$-bundle on $X$ and $\epsilon_{{x}}$ is a $B$-reduction of the induced $G$-bundle $\mathscr{P}_B\underset{B}{\times}G$ at $x$. Similarly, we can define $\Bun_G, \Bun_G'$.

\subsection{Work of \cite{[Camp]}}The singularities of the Drinfeld compactification $\overline{\Bun}_B$ contain important representation theoretical information. In \cite{[Camp]}, the author constructed an algebraic stack $\overline{\Bun}_B^{K}$ which is a resolution of singularities of  $\overline{\Bun}_B$. It turns out that this algebraic stack is very useful to study $\overline{\Bun}_B$. First of all, $\overline{\Bun}_B^{K}$ is proper over $\overline{\Bun}_B$. Secondly, $\overline{\Bun}_B^K$ has a stratification by degree and defect. The projection map $\overline{\Bun}_B^K\to \overline{\Bun}_B$ preserves degree and defect. Thirdly, the complement of the isomorphic locus in $\overline{\Bun}_B^K$ is a normally crossing divisor. This property makes the calculation of singular support of the intersection cohomology ($!$ or $*$-extension) sheaf on $\overline{\Bun}_B^K$ much easier than the corresponding sheaf on $\overline{\Bun}_B$. Here, the sheaf means 'D-module' if $\Char(\mathsf{k})=0$ and '$\ell$-adic sheaf' if $\Char(\mathsf{k})=p>0$ and $p\neq \ell$.

\subsection{Content of this paper}
Fix a point $x\in X$, we consider the algebraic stack $\overline{\Bun}_B'$ which is the Drinfeld compactification of $\Bun_B'$. 

Rather than $\overline{\Bun}_B'$, it is more interesting to study substacks of $\overline{\Bun}_B'$.  The algebraic stack $\Bun_B'$ admits a stratification $\{\Bun_B^w,\ w\in W\}$ (see Section \ref{rela} for the definition) by the relative position of the restriction of the $B$-bundle and the Iwahori structure. We are particularly interested in the geometry of the closure of ${\Bun}^w_B$ in $\overline{\Bun}_B'$. The singularities of $\Bun_B^w$ in its closure $\overline{\Bun}_B^w$ contain representation theory information different from the one that we can get from the singularities of $\overline{\Bun}_B$.

In order to define a resolution of singularities of $\overline{\Bun}^w_B$, first, we construct a smooth algebraic stack $\overline{\Bun}^{K,'}_{B,x}$. It is a resolution of singularities of $\overline{\Bun}_B'$. It enjoys favorable properties as $\overline{\Bun}_B^K$. The important note here is that $\overline{\Bun}^{K,'}_{B,x}$ admits a relative position map \eqref{eq 3.2} to $B\backslash G/B$. 

Based on the algebraic stack $\overline{\Bun}^{K,'}_{B,x}$, we construct an algebraic stack $\BS^w$ using the Bott-Samelson resolution in Section \ref{def 3.3.1}. We will show that $\BS^w$ is the resolution of singularities of $\overline{\Bun}_B^{w}$ when the degree is sufficiently antidominant with respect to the defect, and we will see that $\BS^w$ enjoys many favorable properties as $\overline{\Bun}_B^K$. 

Then, in Section \ref{sect 4.2}, we will use $\BS^w$ to prove the universally local acyclicity of the $!$-extension sheaf $j_{!, \overline{\Bun}_B^w}$ on $\overline{\Bun}_B^{w}$. With the same argument, we can also prove the universally local acyclicity of the $*$-extension sheaf and the intersection cohomology sheaf on $\overline{\Bun}_B^{w}$.

\subsection{Relation with \cite{[Camp]}}
This paper considers a refinement of \cite{[Camp]} in the sense of Iwahori Eisenstein series functor. Given a sheaf $\mathcal{G}$ on $\overline{\Bun}_B'$ as the {kernel}, the Iwahori {Eisenstein} series functor $\Eis_{\mathcal{G}}$ associated with $\mathcal{G}$ is defined as \[\Eis_{\mathcal{G}}(\mathcal{F}):=p_{T,*}(p_G^!(\mathcal{F})\overset{!}{\otimes} \mathcal{G}).\] Here, $p_G, p_T$ are projections from $\overline{\Bun}_B'$ to $\overline{\Bun}_G'$ and ${\Bun}_T$, respectively.

The pullback of the intersection cohomology sheaf ${\IC_{\overline{\Bun}_B}}$ on $\overline{\Bun}_B$ along the projection \[\overline{\Bun}_B'\longrightarrow \overline{\Bun}_B\] is the intersection cohomology sheaf ${\IC_{\overline{\Bun}{}_B'}}$ on ${\overline{\Bun}{}_B'}$.  Hence, the composition of the usual Eseinstein series functor associated with ${\IC_{\overline{\Bun}_B}}$ and the pullback functor $\Shv(\Bun_G)\longrightarrow \Shv(\Bun_G')$ is isomorphic to the Iwahori Eseinstein series functor associated with ${\IC_{\overline{\Bun}{}_B'}}$. The work of \cite{[Camp]} provides a powerful tool to study this functor. This functor admits a filtration by the Iwahori Eseinstein series functors associated with the sheaves  $j_{!, \overline{\Bun}_B^w}$  studied in this paper. 

The constructions and proofs used in this paper are similar to those of \cite{[Camp]}.
\subsection{Application}
The application of $j_{!, \overline{\Bun}_B^w}$ has not been explored, but it has already shown its importance in the fundamental local equivalence of the geometric Langlands program proposed by D. Gaitsgory and J. Lurie.

Roughly speaking, using $j_{!, \overline{\Bun}_B^w}$ (more precisely, its twisted version $j^c_{!, \overline{\Bun}_B^w}$) as a kernel, we can construct an equivalence functor from the twisted Whittaker category on the affine flags to {$\Rep_q^\mix(\check{G})$, the category $\mathcal{O}$ of the mixed quantum group.} 

More precisely, let us denote by $(\overline{\Bun}_{N^-}')_{\infty\cdot x}$ the Drinfeld compactification of the algebraic stack $\Bun_{N^-}'$ with a possible pole at $x\in X$. As same as \cite[Section 4.2.2]{[FGV]}, we can define $\Whit^c((\overline{\Bun}_{N^-}')_{\infty\cdot x})$, the category of twisted Whittaker sheaves on the affine flags, as a full subcategory of the category of twisted sheaves on $(\overline{\Bun}'_{N^-})_{\infty\cdot x}$ by imposing an equivariance condition with respect to a certain unipotent groupoid against a non-degenerate character.

By a construction similar to \cite[Section 2]{[BFGM]}, we define the compact Iwahori Zastava space $\mathcal{Z}$ as an open substack of the {fiber} product of $(\overline{\Bun}'_{N^-})_{\infty\cdot x}$ and $\overline{\Bun}_B'$ over $\Bun_G'$ with the condition that the generic $N^-$-reduction and the generic $B$-reduction are transverse. It is known that the compact Iwahori Zastava space admits an ind-proper morphism to the configuration space $\Conf_x$ of coweight-colored divisors {$D$} on $X${,  such that $\langle \check{\lambda}, D\rangle$ is anti-effective away from $x$ for any dominant weight $\check{\lambda}$ of $G$}.

Consider the following diagram,
\begin{equation}\label{con dia}
    \xymatrix{
&\mathcal{Z}\ar[ld]_{\textnormal{p}_{\Whit}}\ar[rd]^v&\\
(\overline{\Bun}_{N^-})_{\infty\cdot x}'&& \Conf_x.}
\end{equation}
The $!$-pullback of $j^c_{!, \overline{\Bun}_B^{w_0}}$ to $\mathcal{Z}$ gives a kernel for a functor {$F$} from $\Whit^c((\overline{\Bun}_{N^-}')_{\infty\cdot x})$ to the category of twisted sheaves on $\Conf_x$. {Let $\Omega_q^{\Lus}$ be $F(L^0)$, here $L^0$ is the pullback of the unit Whittaker sheaf on the Drinfeld compactification of $\Bun_{N^-}$. One can check that $\Omega_q^\Lus$ is a factorization algebra, and image of $F$ acquires a factorization module structure over $\Omega_q^\Lus$, so $F$ induces a functor $F_1: \Whit^c((\overline{\Bun}_{N^-}')_{\infty\cdot x})\to \Omega_q^{\Lus}-\Fact$, here $q=\exp(\pi i c)$. The target category is the category of factorization modules on $\Conf_x$ over $\Omega_q^\Lus$, and is equivalent to $\Rep_q^\mix(\check{G})$, ref. {\cite[Theorem 1.2.1]{[CF]}}. 


The main theorem of \cite{[Y]} claims that $F_1$ is an equivalence. The key step in \cite{[Y]} is to show that the standard object $\Delta_\lambda$ in the Whittaker category goes to standard object in $\Omega_q^{\Lus}-\Fact$, for any weight $\lambda$ of $\check{G}$. However, it is very difficult to calculate $F_1(\Delta_\lambda)$ directly. } 

{The main theorem of this paper (Theorem \ref{Kont}) solves the above issue. Namely, } {o}ne can also construct a functor $F_2$ using the $*$-extension sheaf $j^{{-}c}_{*, \overline{\Bun}_B^{w_0}}$ as the kernel. The main theorem of this paper implies that $F_1$ and $F_2$ are Verdier dual to each other. {So, we have $F_1(\Delta_\lambda)=\BD F_2(\BD(\Delta_\lambda))$. The object $F_2(\BD(\Delta_\lambda))$ is easier to calculate, and it is shown in \cite{[Y]} that $F_2(\BD(\Delta_\lambda))$ is the unique factorization module on $\Conf_x$ over $\BD(\Omega_q^{\Lus})$ whose $!$-stalks at $\mu\cdot x\in \Conf_x$ is 1-dimensional if $\lambda=\mu$, and is $0$ otherwise. In particular, its Verdier dual is the standard object in  $\Omega_q^{\Lus}-\Fact$.} 

\section{Kontsevich compactification}
\subsection{Drinfeld compactification}
The Drinfeld compactification of the moduli stack of principal bundles is introduced in \cite{[BG]}. First of all, let us recall the definition of the Drinfeld compactification of $\Bun_B$ and $\Bun_B'$. Let $\check{\Lambda}^+$ denote the set of dominant weights of $G$.
\begin{definition}\label{2.1.1}
The algebraic stack $\overline{\Bun}_B$ classifies the data $(\mathscr{P}_T,\mathscr{P}_G, \{\kappa^{\check{\lambda}}, \check{\lambda}\in \check{\Lambda}^+\})$, where
\begin{enumerate}[label=(\alph*)]
    \item $\mathscr{P}_T$ is a $T$-bundle on $X$,
    \item $\mathscr{P}_G$ is a $G$-bundle on $X$,
    \item $\kappa^{\check{\lambda}}: \check{\lambda}(\mathscr{P}_T)\longrightarrow \mathcal{V}^{\check{\lambda}}_{\mathscr{P}_G}$ is a collection of {monomorphisms} of coherent sheaves satisfying the Pl\"{u}cker relations (see \cite[Section 1]{[BG]}).
\end{enumerate}
Here, $\check{\lambda}(\mathscr{P}_T)$ is the induced line bundle on $X$, and $\mathcal{V}^{\check{\lambda}}_{\mathscr{P}_G}$ is the vector bundle associated with $\mathscr{P}_G$ with the fiber Weyl module $\mathcal{V}_G^{\check{\lambda}}$.

 Similarly, we denote by $\overline{\Bun}_B'$ the algebraic stack which classifies the data $(\mathscr{P}_T,\mathscr{P}_G, \{\kappa^{\check{\lambda}}, \check{\lambda}\in \check{\Lambda}^+\}, \epsilon_{{x}})$. Here $\mathscr{P}_T,\mathscr{P}_G, \textnormal{and}\ \{\kappa^{\check{\lambda}}, \check{\lambda}\in \check{\Lambda}^+\}$ are as above, $\epsilon_{{x}}$ is a $B$-reduction of $\mathscr{P}_G$ at $x\in X$.
\end{definition}

We call $\overline{\Bun}_B$ (resp. $\overline{\Bun}_B'$) the Drinfeld compactification of $\Bun_B$ (resp. $\Bun_B'$). 

If we require $\kappa^{\check{\lambda}}$ to be injective {bundle maps (i.e., cokernel of each $\kappa^{\check{\lambda}}$ is $\cO_X$-flat)} in Definition \ref{2.1.1}, the resulting algebraic stack is exactly the algebraic stack  $\Bun_B$ (resp. $\Bun_B'$). The algebraic stack $\Bun_B$ (resp. $\Bun_B'$) is open dense in  $\overline{\Bun}_B$ (resp. $\overline{\Bun}_B'$).


\subsection{Stratification by degree and defect}
Let $\Lambda$ be the coweight lattice of $G$, and $\Lambda^{pos}$ be the semi-group generated by positive coroots.

For $\lambda\in \Lambda$, $\mu\in \Lambda^{pos}$, we denote by $\overline{\Bun}^{\lambda}_{B,\leq \mu}$ the open substack of $\overline{\Bun}_B$ such that the degree of $\mathscr{P}_T$ is $\lambda$ and the total defect {(ref., \cite[Proposition 1.2.5, Section 3.3]{[BG]}, \cite[7.7]{[G1]})} of the generalized $B$-structure is no more than $\mu$. We define $\overline{\Bun}^{',\lambda}_{B,\leq \mu}$ as the fiber product of $\overline{\Bun}^{\lambda}_{B,\leq \mu}$ and $\Bun_G'$ over $\Bun_G$.


\subsection{Kontsevich compactification}
For a connected projective smooth curve $X$ with genus $g$, we will construct a Kontsevich compactification $\overline{\Bun}^{K,'}_{B,x}$ of ${\Bun}_B'$ in Section \ref{sect 2.5}. This stack provides a resolution of singularities of $\overline{\Bun}_B'$. Before we define $\overline{\Bun}^{K,'}_{B,x}$, let us introduce a closely related but larger stack $\overline{\Bun}_B^{K,'}$. 
\begin{definition}\label{def C}
Fix a point $x\in X$, we denote by $\overline{\Bun}_B^{K,'}$ the stack whose $S$-points classify the data ($C, p, \mathscr{P}_G, {\epsilon_{C}}, \epsilon_{{x}}, s$), where

\begin{enumerate}[label=(\alph*)]
    \item $C$ is a flat $S$-family of nodal projective curves over $X$ {of arithmetic genus $g$},
    \item $p$ is a map from $C$ to $X\times S$  
    \item $\mathscr{P}_G$ is a $G$-bundle on $X\times S$,
    \item ${\epsilon_{C}}$ is a $B$-reduction of $p^*\mathscr{P}_G$ on $C$, i.e., a subset $\mathscr{P}_B$ of $p^*\mathscr{P}_G$ which is also a $B$-bundle  on $C$,
    \item $\epsilon_{{x}}$ is a $B$-reduction of $\mathscr{P}_G$ on $x\times S$,
    \item $s$ is a section of ${C\rightarrow}S$ in the smooth locus of $C$,
\end{enumerate}

with the following conditions:

\begin{enumerate}[label=(\arabic*)]
    \item the map \[C\to (G/B)_{\mathscr{P}_G}\] is stable over any geometric point in $S$, i.e., the automorphism group  in the moduli space of curves with a marked point mapping to $(G/B)_{\mathscr{P}_G}$ is finite,
    \item the map $p\colon C\longrightarrow X\times S$ is of degree 1.
\end{enumerate}                                 
\end{definition}

\begin{rem} The Kontsevich compactification $\overline{\Bun}_B^K$ defined in \cite{[Camp]} classifies the data (a)-(d), with the stability condition (without a marked point) and the degree condition (2). It provides a resolution of singularities of $\overline{\Bun}_B$. In particular, its fiber product with the classifying stack $\pt/B$ over the classifying stack $\pt/G$ provides a resolution of singularities of $\overline{\Bun}_B'$. However, the author of this paper cannot construct a resolution of singularities of $\overline{\Bun}_B^w$ with $\overline{\Bun}_B^K\underset{\pt/G}{\times}{\pt/B}$.
\end{rem}
\begin{rem}
\label{11.3.6}In the definition of $\overline{\Bun}_B^{K,'}$, if we require that the map $p$ is an isomorphism, the resulting stack is isomorphic to $\Bun'_B\times X$. 
\end{rem}

The assignment sending ($C, p, \mathscr{P}_G, {\epsilon_{C}}, \epsilon_{{x}}, s$) to ($\mathscr{P}_G, \epsilon_{{x}}$) gives rise to a map from $\overline{\Bun}_B^{K,'}$ to $\Bun_G'$.
\begin{prop}\label{0.1}
$\overline{\Bun}_B^{K,'}$ is an algebraic stack locally of finite type, and is proper over $\Bun_G'$.
\end{prop}

\begin{proof}
Given a $S$-point $(\mathscr{P}_G, \epsilon_{{x}})$ of $\Bun_G'$, i.e., a map $S\longrightarrow \Bun_G'$. We consider the fiber product of $\overline{\Bun}_B^{K,'}$ and $S$ over $\Bun_G'$. By definition, it is an open and closed locus of the moduli stack $\overline{\mathscr{M}_{g,1}}((G/B)_{\mathscr{P}_G}):=\{h\colon (C,s)\longrightarrow (G/B)_{\mathscr{P}_G} \}$ of stable maps defined on arithmetic genus $g$ curves with a marked point, such that the composition map $C\longrightarrow(G/B)_{\mathscr{P}_G} \longrightarrow X\times S$ has degree one. By \cite{[AO]}, $\overline{\mathscr{M}_{g,1}}((G/B)_{\mathscr{P}_G})$ is a proper Deligne-Mumford stack locally of finite type. In particular, $\overline{\Bun}_B^{K,'}$ is an algebraic stack locally of finite type and is proper over $\Bun_G'$.


\end{proof}

By \cite[Section 3.2]{[Camp]}, we have a map from $\overline{\Bun}_B^K$ to $\overline{\Bun}_B$. This construction naturally lifts to a map \begin{equation}\label{17.12}
   \pr\colon \overline{\Bun}_B^{K,'}\longrightarrow \overline{\Bun}_B'.
\end{equation} By the construction in \cite[Section 3.2]{[Camp]}, the composition \[\overline{\Bun}_B^{K,'}\longrightarrow \overline{\Bun}_B'\longrightarrow \Bun_G'\] coincides with the projection map from $\overline{\Bun}_B^{K,'}$ to $\Bun_G'$. 

\begin{cor}\label{cor 2.1} The morphism $\pr$ is proper.
\end{cor}
\begin{proof}
 The composition of $\overline{\Bun}_B^{K,'}\longrightarrow \overline{\Bun}_B'$ and $\overline{\Bun}_B' \longrightarrow \Bun_G'$ is proper by the above proposition. Now the claim follows from the fact that the map $\overline{\Bun}_B' \longrightarrow \Bun_G'$ is also proper (it is a base change of the proper map $\overline{\Bun}_B\longrightarrow \Bun_G$, see \cite[Proposition 1.2.2]{[BG]}).
\end{proof}


\subsection{Degree and defect of $\overline{\Bun}_B^{K,'}$}
In \cite[Section 3.3]{[Camp]}, the author defined degree and defect for geometric points of $\overline{\Bun}_B^K$. By repeating the definition in $loc. cit$ word-by-word, we define degree and defect of geometric points of $\overline{\Bun}_B^{K,'}$. To be self-contained, we briefly sketch the definitions.

If $S=\textnormal{Spec}\ \mathsf{k}$, since $p$ is of degree 1, $C$ has an irreducible component isomorphic to $X$ such that the restriction of $p$ to it is an isomorphism. Furthermore, because the curves $X$ and $C$ have the same genus $g$, all other components $C_1, C_2,\cdots, C_n$ of $C$ are isomorphic to $\mathbb{P}^1$, and the restriction of $p$ to any $C_i$ is constant. The dual graph of $C$ is a tree. 

\begin{definition}
For a geometric point $(C, p, \mathscr{P}_G, {\epsilon_{C}}, \epsilon_{{x}}, s)$ of $\overline{\Bun}_B^{K,'}$, its degree is defined as the degree of $\mathscr{P}_B$, and its defect is defined as $-\sum_{1\leq i\leq n} \deg(\mathscr{P}_B|_{C_i})$. Here $\mathscr{P}_B$ is the $B$-bundle on $C$ associated with ${\epsilon_{C}}$.
\end{definition} 
By the same proof as that of \cite[Proposition 3.3.1]{[Camp]}, we can prove that the map (\ref{17.12}) preserves degrees and defects for geometric points. 

For $\lambda\in \Lambda$ and $\mu\in \Lambda^{pos}$, there is a unique open algebraic substack locally of finite type of $\overline{\Bun}_B^{K,'}$, such that its geometric points consist of those geometric points $(C, p, \mathscr{P}_G, {\epsilon_{C}}, \epsilon_{{x}}, s)$ with degree $\lambda$ and defect smaller than or equal to $\mu$. Indeed, the desired algebraic stack is the fiber product of $\overline{\Bun}_{B, \leq\mu}^{',\lambda}$ and $\overline{\Bun}_B^{K,'}$ over $\overline{\Bun}_{B}^{'}$. We denote it by $\overline{\Bun}_{B, {\leq \mu}}^{K,',\lambda}$. 


\subsection{Construction of $\overline{\Bun}^{K,'}_{B,x}$}\label{sect 2.5}To construct a resolution of singularities of $\overline{\Bun}_B'$, we need to consider a closed substack of $\overline{\Bun}_B^{K,'}$ rather than itself.

Note that $\overline{\Bun}_B^{K,'}$ admits an evaluation map to $X$ which sends ($C, p, \mathscr{P}_G, {\epsilon_{C}}, \epsilon_{{x}}, s$) to $p(s)$.

\begin{definition}
For $x\in X$, we define $\overline{\Bun}^{K,'}_{B,x}$ as the fiber product of $\overline{\Bun}_B^{K,'}$ and ${x}$ over $X$.
\end{definition}
\begin{rem}
It does not matter if this fiber product is defined in the usual algebraic geometry sense or the derived algebraic geometry sense, because we will see that $\overline{\Bun}_B^{K,'}$ is actually smooth over $X$ (in this case, derived fiber product is classical). 
\end{rem}

\begin{definition}
Let $\mathscr{M}'_X$ denote the moduli stack classifying the data $(C, p, s)$, where
\begin{enumerate}[label=(\alph*)]
    \item $C$ is a connected nodal curve,
    \item $s$ is a point in the smooth locus of $C$,
    \item $p: C\longrightarrow X$ is a proper morphism, such that $p$ is of degree 1, and $R^1 p_*{O}_C=0$.
    \end{enumerate}
\end{definition}
\begin{rem}
The moduli stack $\mathscr{M}_X$ defined in \cite[Section 2.4]{[Camp]} classifies the data $(a)$ and $(c)$.
\end{rem}
\begin{rem}
It is not necessary to assume that $X$ is projective in the definition.
\end{rem}



Applying the same method as \cite[Propsition 2.4.1]{[Camp]}, we have the following smoothness property of $\mathscr{M}_X'$.
\begin{prop}\label{cor 2.2}
The moduli stack $\mathscr{M}'_X$ is smooth, and $\overline{\Bun}_B^{K,'}$ is smooth over $\mathscr{M}_X'$. 
\end{prop}
\begin{proof}
To prove the smoothness of $\mathscr{M}_X'$, we should show that the cohomology of its tangent complex vanishes for strictly positive degree.

Note that the tangent complex of $\mathscr{M}_X'$ at $(C,p,s)$ is given by {$R\Gamma(C,\CK[1])$, where $\CK$ fits into the distinguished triangle
\begin{equation}\label{2.1 seq}
    \CK\longrightarrow {T}(C)(-s)\longrightarrow p^*{T}(X)\overset{+1}{\longrightarrow}.
\end{equation}
} 

By a direct calculation, the cohomology $R^i\Gamma(C, {T}(C)(-s))$ is concentrated in degrees 0 and 1. Furthermore, {since $p^* T(X)$ is a line bundle,} the cohomology $R^i\Gamma(C, p^*{T}(X))$ vanishes {unless $i=0,1$}. According to the long exact sequence associated with \eqref{2.1 seq}, we have $R^i\Gamma(C, {\CK}{[1]})=0$ if $i> 1$. We only need to prove 
\[R^1\Gamma(C, {T}(C)(-s))\longrightarrow R^1\Gamma(C, p^*{T}(X))\] is surjective. It follows from the fact that the above morphism factors through $R^1\Gamma(C, {T}(C))$, and \[R^1\Gamma(C, {T}(C))\longrightarrow R^1\Gamma(C, p^*{T}(X))\] is surjective by the proof of \cite[Proposition 2.4.1]{[Camp]}.

Let us prove that $\overline{\Bun}_B^{K,'}$ is smooth over $\mathscr{M}_X'$. Given a point $(C,p,s)\in \mathscr{M}_X'$, its fiber in $\overline{\Bun}_B^{K,'}$ is the open locus of the fiber product 
\[\Bun_G'\underset{\Maps(C, \pt/G)}{\times}\Maps(C, \pt/B)\]
with the stability condition. Here $\Maps(-,-)$ denotes the mapping stack.

Since $C$ is a curve, $\Maps(C, \pt/B)$ is smooth. We only need to show that $\Bun_G'$ is smooth over $\Maps(C, \pt/G)$. Now the claim follows from the fact that the maps $\Bun_G'\longrightarrow \Bun_G$ and $\Bun_G\longrightarrow \Maps(C,\pt/G)$ are smooth (\cite[Proposition 2.4.1]{[Camp]}). 
\end{proof}
\begin{rem}
Since $\mathscr{M}_X'$ is smooth and it has a deformation along $X$, $\mathscr{M}_X'$ is smooth over $X$. Also, by the above proposition, $\overline{\Bun}_B^{K,'}$ is smooth over $X$. 
\end{rem}
\begin{prop}
The moduli stack $\mathscr{M}_X'$ is smooth over $\mathscr{M}_X$.
\end{prop}

\begin{proof}
By \cite[Proposition 2.4.1]{[Camp]}, the moduli stack $\mathscr{M}_X$ is smooth. Since $\mathscr{M}_X'$ is also smooth, we only need to prove that the tangent map is surjective on $R^0\Gamma$. 

Note that the tangent complex of $\mathscr{M}_X$ at $(C,s)$ is {$R\Gamma(C, \CK'[1])$, where $\CK'\colon= T(C/X)$ and fits into the distinguished triangle 
\begin{equation}\label{2.2 seq}
    \CK'\longrightarrow {T}(C)\longrightarrow p^*{T}(X)\overset{+1}{\longrightarrow}.
\end{equation}

}

We should show that $R^1(C, {\CK})\longrightarrow R^1(C, {\CK'})$ is surjective.

By long exact sequence, we have the following commutative diagram of exact sequences
\[
\xymatrix{
R^0\Gamma(C,p^*{T}(X))\ar[r]\ar@{=}[d]&R^1\Gamma(C, \mathcal{K})\ar[r]\ar[d]&R^1\Gamma(C,{T}{(}C{)}(-s))\ar[r]\ar@{->>}[d]&R^1\Gamma(C,p^*{T}(X))\ar@{=}[d]\\
R^0\Gamma(C,p^*{T}(X))\ar[r]&R^1\Gamma(C, \mathcal{K}')\ar[r]&R^1\Gamma(C,{T}{(}C{)})\ar[r]&R^1\Gamma(C,p^*{T}(X)).}
\]
Applying the four lemma to the above diagram, we obtain the desired claim.
\end{proof}
\begin{definition}
Let $\mathscr{M}_{X,x}'$ be the fiber product 
\[\mathscr{M}_X'\underset{X}{\times}x.\]
\end{definition}
Because $\mathscr{M}_X'$ is smooth over $X$, $\mathscr{M}_{X,x}'$ is smooth.
\begin{cor}\label{cor bef 2.5}
$\overline{\Bun}^{K,'}_{B,x}$ is smooth over $\mathscr{M}_{X,x}'$. In particular, $\overline{\Bun}^{K,'}_{B,x}$ is smooth.
\end{cor}
\subsection{Normal crossing property}
It is proved in \cite[Proposition 4.4.1]{[Camp]} that the complement $\partial \mathscr{M}_X$ of the open dense substack $(C=X, p=\id)$ of $\mathscr{M}_X$ is a normal crossing divisor, and its preimage in $\overline{\Bun}_B^K$ is exactly the complement of $\Bun_B$. In this section, we will prove that $\mathscr{M}_{X,x}'$ enjoys similar properties. 

The moduli stack $\mathscr{M}_{X,x}'$ admits a stratification by the pattern of the dual graph of $C$, the dual graph of $C\underset{X}{\times}\{x\}$, and the place of the marked point. The open dense stratum is the point $\{C=X, p=\id, s=x\}$. Denote by $\partial \mathscr{M}_{X,x}'$ the boundary of $\{C=X, p=\id, s=x\}$ in $\mathscr{M}_{X,x}'$.

Applying the method of \cite[Proposition 4.4.1]{[Camp]}, we prove that $\partial \mathscr{M}_{X,x}'$ is a normal crossing divisor.
\begin{lem}
$\partial \mathscr{M}_{X,x}'$ is a normal crossing divisor in $\mathscr{M}_{X,x}'$.
\end{lem}
\begin{proof}
 Note that in the definition of $\mathscr{M}_{X,x}'$, we do not need to require $X$ to be projective, the definition makes sense for any smooth connected curve with a fixed point. Now let us first reduce the proof to the case $X=\mathbb{A}^1$ and $x=0$.

Given $(C,p,s)\in \mathscr{M}_{X,x}'$, let $U$ be an open subvariety of $X$ which contains $x$ and all nodal points $x_1, \cdots, x_n$ of $C$ on $X$. The restriction gives rise to an \'{e}tale map \[\mathscr{M}_{X,x}'\longrightarrow \mathscr{M}_{U,x}'.\]

Let $f\colon U\longrightarrow \mathbb{A}^1$ be an \'{e}tale morphism such that $f^{-1}(0)=\{x\}$ and $f^{-1}(f(x_i))=\{x_i\}$ for any $x_i$. It induces a morphism \[\mathscr{M}_{\mathbb{A}^1,0}'\longrightarrow \mathscr{M}_{U,x}'\] by sending $(\widetilde{C}, \widetilde{p}, \widetilde{s})\in \mathscr{M}_{\mathbb{A}^1,0}'$ to $(\widetilde{C}\underset{\mathbb{A}^1}{\times}U, \widetilde{p}', \widetilde{s}') \in \mathscr{M}_{U,x}'$ where $\widetilde{p}'\colon \widetilde{C}\underset{\mathbb{A}^1}{\times}U\longrightarrow U\times S$ is the base change of $\widetilde{p}$, and $\widetilde{s}'$ is identified with $\widetilde{s}$ via the identification $\widetilde{C}\underset{\mathbb{A}^1}{\times}\{0\}= \widetilde{C}\underset{\mathbb{A}^1}{\times}U\underset{U}{\times}\{x\}$. 

The image of $(C,p,s)$ in $\mathscr{M}_{U,x}'$ under the morphism $\mathscr{M}_{X,x}'\longrightarrow \mathscr{M}_{U,x}'$ lifts to a point $(\widetilde{C},\widetilde{p},\widetilde{s})\in \mathscr{M}_{\mathbb{A}^1,0}$. By the assumption on $f$, the morphism  $\mathscr{M}_{\mathbb{A}^1,0}'\longrightarrow \mathscr{M}_{U,x}'$ is \'{e}tale at  $(\widetilde{C},\widetilde{p},\widetilde{s})\in \mathscr{M}_{\mathbb{A}^1,0}'$. Since the normal crossing property is \'{e}tale invariant, it suffices to prove that $\partial \mathscr{M}_{\mathbb{A}^1,0}'$ is a normal crossing divisor.

{Consider the map 
\begin{equation}
    \begin{split}
\mathscr{M}_{\mathbb{A}^1}'\longrightarrow \mathscr{M}_{\mathbb{A}^1}'\times& \mathbb{A}^1\longrightarrow \mathscr{M}_{\mathbb{A}^1}'\times \mathbb{A}^1\\
(C,p,s)\mapsto ((C,p,s),& p(s))\mapsto ((C,p-p(s),s), p(s)),
    \end{split}
\end{equation}
here $p-p(s): C\to \mathbb{A}^1$ is the composed map $C\overset{p}{\to}\BA^1\overset{a\mapsto a-p(s)}{\to} \BA^1.$
} Note that {the above map induces an isomorphism} $\mathscr{M}_{\mathbb{A}^1}'{\simeq \mathscr{M}_{\mathbb{A}^1, 0}'\times \mathbb{A}^1.}$ {W}e have to prove that the complement of $\{C=\mathbb{A}^1, p=\id\}\times \mathbb{A}^1$ in $\mathscr{M}_{\mathbb{A}^1}'$ is a normal crossing divisor. Since $\mathscr{M}_{\mathbb{A}^1}'$ is smooth over $\mathscr{M}_{\mathbb{A}^1}$ and $\partial \mathscr{M}_{\mathbb{A}^1}'$ equals the preimage of $\partial \mathscr{M}_{\mathbb{A}^1}:= \mathscr{M}_{\mathbb{A}^1}\setminus \{C=\mathbb{A}^1, p=\id\}$. Now the claim follows from the facts that smooth pullback of a normal crossing divisor is still a normal crossing divisor and $\partial \mathscr{M}_{\mathbb{A}^1}$ is a normal crossing divisor. 
\end{proof}
\begin{rem}
A stratum of $\mathscr{M}_{X,x}'$ with $q$-nodal points is of codimension $q$. For each point in this stratum, there are $q$ (not necessarily different) codimension 1 components  of $\partial \mathscr{M}_{X,x}'$ intersect transversally at this point, each component is obtained by smoothening $q-1$ nodal points out of $q$.
\end{rem}
\begin{cor}
 The boundary of $\Bun_B'$ in $\overline{\Bun}^{K,'}_{B,x}$ is a normal crossing divisor.
\end{cor}
\begin{proof}
The claim follows from Corollary \ref{cor bef 2.5} and the observation that the preimage of $\partial \mathscr{M}_{X,x}'$ along $\overline{\Bun}^{K,'}_{B,x}\longrightarrow \mathscr{M}_{X,x}'$ is exactly the complement of $\Bun_B'$ in $\overline{\Bun}^{K,'}_{B,x}$.
\end{proof}
By Corollary \ref{cor 2.1} and Corollary \ref{cor bef 2.5}, the algebraic stack $\overline{\Bun}^{K,'}_{B,x}$ is a smooth algebraic stack and it is proper over $\overline{\Bun}'_{B}$. Furthermore, the restriction of $\pr$ to $\Bun_B'$ is an isomorphism. Hence, $\overline{\Bun}^{K,'}_{B,x}$ provides a resolution of singularities of $\overline{\Bun}_B'$ with a normal crossing boundary.
\subsection{Smoothness lemma} 
In this section, we introduce a smoothness lemma which is useful in Section \ref{3} and \ref{4}. From now on, suppose $\lambda$ is sufficiently antidominant relative to $\mu$. To be more precise, we assume $\lambda$ and $\mu$ satisfy the following condition.
\begin{cond}\label{X}
For any negative root $\check{\alpha}^-$ of $G$ {and any $0\leq \mu'\leq \mu$}, we have    $\langle\check{\alpha}^-,\lambda+{\mu'}\rangle>2g-1 $. 
\end{cond}

In the definition of $\overline{\Bun}^{K,'}_{B,x}$, we require that the section $s$ is over $x$. As a result, ${\epsilon_{C}}|_s$ the restriction of ${\epsilon_{C}}$ to $s$ gives a $B$-reduction of $\mathscr{P}_G|_x$. In particular, there is a map of prestacks \[\overline{\Bun}^{K,'}_{B,x}\longrightarrow \Bun_G'':=\Bun_G'\underset{\pt/G}{\times}\pt/B\] which sends $(C, p, \mathscr{P}_G, {\epsilon_{C}}, \epsilon_{{x}}, s)$ to $(\mathscr{P}_G, \epsilon_{{x}}, {\epsilon_{C}}|_s)$. 

Let $\overline{\Bun}^{K,',\lambda}_{B,x,\leq \mu}$ be the open substack of $\overline{\Bun}^{K,'}_{B,x}$ such that the degree is $\lambda$ and the defect is smaller than or equal to $\mu$. In other words, $\overline{\Bun}^{K,',\lambda}_{B,x,\leq \mu}$ is the intersection of $\overline{\Bun}^{K,'}_{B,x}$ and $\overline{\Bun}_{B, {\leq \mu}}^{K,',\lambda}$. The following proposition is the analog of \cite[Proposition 4.1.1]{[Camp]} in our case.
\begin{prop}\label{pro l}
If $\lambda$ and $\mu$ satisfy Condition \ref{X}, then the following map
\begin{equation}\label{l}
\begin{split}
 \overline{\Bun}^{K,',\lambda}_{B,x,\leq \mu}\qquad&\longrightarrow \mathscr{M}'_{X,x}\times \Bun_G''\\
(C, p, \mathscr{P}_G, {\epsilon_{C}}, \epsilon_{{x}}, s)&\mapsto (C, p, s), (\mathscr{P}_G, \epsilon_{{x}}, {\epsilon_{C}}|_s)
\end{split}
\end{equation}
is smooth.

\end{prop}

\begin{proof}
Given a point $((C, p, s), (\mathscr{P}_G, \epsilon_{{x}}, \epsilon_{{x}}'))$ in $\mathscr{M}_{X,x}'\times \Bun_G''$, its preimage in $\overline{\Bun}^{K,',\lambda}_{B,x,\leq \mu}$ is the open locus of \[\Maps_X(C, (G/B)_{\mathscr{P}_G})\mathop{\times}\limits_{\Maps_X(s, (G/B)_{\mathscr{P}_G})} \pt\] with the stability condition and the degree-defect condition. Here $\pt$ is the point scheme, and the map $\pt\longrightarrow \Maps_X(s, (G/B)_{\mathscr{P}_G})$ is given by ${\epsilon_{C}}|_s$.

We should prove that  $\Maps_X(C, (G/B)_{\mathscr{P}_G})$ is smooth over $\Maps_X(s, (G/B)_{\mathscr{P}_G})$ at any point in $\Maps_X(C, (G/B)_{\mathscr{P}_G})$ with the stability condition and the degree-defect condition. Since $\Maps_X(C, (G/B)_{\mathscr{P}_G})$ and $\Maps_X(s, (G/B)_{\mathscr{P}_G})\simeq (G/B)_{\mathscr{P}_G|_x}$ are smooth (ref. \cite[Proposition 4.1.1]{[Camp]}), we only need to prove that the tangent map \[{T}\Maps_X(C, (G/B)_{\mathscr{P}_G})\longrightarrow {T}(G/B)_{\mathscr{P}_G|_x}\] induces a surjective map on $R^0\Gamma$.

Indeed, given a point in $\Maps_X(C, (G/B)_{\mathscr{P}_G})$, i.e., a morphism $f: C\longrightarrow (G/B)_{\mathscr{P}_G}$. The tangent complex ${T}_{f}\Maps_X(C, (G/B)_{\mathscr{P}_G})$ is given by
\begin{equation}
    R\Gamma(C, f^*{T}((G/B)_{\mathscr{P}_G}/X)).
\end{equation}
Due to the isomorphism:
$f^*{T}((G/B)_{\mathscr{P}_G}/X)\simeq  (\mathfrak{g}/\mathfrak{b})_{\mathscr{P}_B}$, we have
\begin{equation}
    {T}_{f}\Maps_X(C, (G/B)_{\mathscr{P}_G})=R\Gamma(C, (\mathfrak{g}/\mathfrak{b})_{\mathscr{P}_B}).
    \end{equation}
Similarly, the tangent complex of ${T}(G/B)_{\mathscr{P}_G|_x}$ at ${\epsilon_{C}}|_s$ is given by $(\mathfrak{g}/\mathfrak{b})_{\mathscr{P}_B|_s}$. Here $\mathscr{P}_B$ is the $B$-bundle on $C$ associated with ${\epsilon_{C}}$.

We only need to prove the map of restriction of global sections to $s$ 
\begin{equation}\label{2.4}
    R^0\Gamma(C, (\mathfrak{g}/\mathfrak{b})_{\mathscr{P}_B})\longrightarrow (\mathfrak{g}/\mathfrak{b})_{\mathscr{P}_B|_s}
\end{equation}
is surjective. 

First of all, we claim that $(\mathfrak{g}/\mathfrak{b})_{\mathscr{P}_B}$ is globally generated on any irreducible component of $C$. 

If we denote by $\lambda_X$  the degree of $\mathscr{P}_B$ restricted to $X$, we have {$0\leq \lambda_X-\lambda\leq \mu$.} By Condition \ref{X}, 
$\langle\lambda_X,\check{\alpha}^-\rangle{=\langle\lambda+(\lambda_X-\lambda),\check{\alpha}^-\rangle}>2g-1$. In particular, $\check{\alpha}^-(\mathscr{P}_T)|_{X}$ is globally generated. Here $\mathscr{P}_T$ denotes the $T$-bundle induced from $\mathscr{P}_B$. Furthermore,  $(\mathfrak{g}/\mathfrak{b})_{\mathscr{P}_B}$ has a filtration such that the associated graded vector bundle is $\bigoplus_{\check{\alpha}^-\in \check{\Delta}^-} \check{\alpha}^-(\mathscr{P}_T)$. Hence, $(\mathfrak{g}/\mathfrak{b})_{\mathscr{P}_B}|_X$ is globally generated. On any irreducible component $C_i$ except $X$, because the degree of $\mathscr{P}_B|_{C_i}$ is negative, $\check{\alpha}^-(\mathscr{P}_T)$ is globally generated on $C_i\simeq \mathbb{P}^1$. It follows that $(\mathfrak{g}/\mathfrak{b})_{\mathscr{P}_B|_{C_i}}$ is globally generated. 

Now we induct on the number of irreducible components of $C$ to prove the surjectivity of \eqref{2.4}. If $n=0$, i.e., $C=X$, we have already proved \eqref{2.4} is surjective by the globally generated property.

Assume we proved the surjectivity for $n-1$, we choose an irreducible component $C_0$ in $C$ which contains a single nodal point $c$. Let $C'= C\setminus (C_0\setminus\{c\})$.

If $s\in C'$, we consider the following short exact sequence:
\[0\longrightarrow (\mathfrak{g}/\mathfrak{b})_{\mathscr{P}_B}|_{C_0(-c)}\longrightarrow (\mathfrak{g}/\mathfrak{b})_{\mathscr{P}_B}\longrightarrow (\mathfrak{g}/\mathfrak{b})_{\mathscr{P}_B}|_{C'}\longrightarrow 0.\]
Due to  $(\mathfrak{g}/\mathfrak{b})_{\mathscr{P}_B|_{C_0}}$ is globally generated, $H^1(C_0,(\mathfrak{g}/\mathfrak{b})_{\mathscr{P}_B}|_{C_0(-c)})=0$. So the map \[H^0(C,(\mathfrak{g}/\mathfrak{b})_{\mathscr{P}_B})\longrightarrow H^0(C',(\mathfrak{g}/\mathfrak{b})_{\mathscr{P}_B}|_{C'})\] is surjective. According to the inductive hypothesis: the claim is true for $C'$, i.e, \[H^0(C',(\mathfrak{g}/\mathfrak{b})_{\mathscr{P}_B}|_{C'})\longrightarrow (\mathfrak{g}/\mathfrak{b})_{\mathscr{P}_B}|_s\] is surjective. Now the claim follows from the observation that the morphism $H^0(C,(\mathfrak{g}/\mathfrak{b})_{\mathscr{P}_B})\longrightarrow (\mathfrak{g}/\mathfrak{b})_{\mathscr{P}_B}|_s$ factors through $H^0(C',(\mathfrak{g}/\mathfrak{b})_{\mathscr{P}_B}|_{C'})$. 

If $s$ is contained in $C_0$. Applying the above argument to other irreducible components of $C$ except $C_0$, one can show that the map \[H^0(C,(\mathfrak{g}/\mathfrak{b})_{\mathscr{P}_B})\longrightarrow H^0(C_0,(\mathfrak{g}/\mathfrak{b})_{\mathscr{P}_B}|_{C_0})\] is surjective. Now, the claim follows from the globally generated property of $(\mathfrak{g}/\mathfrak{b})_{\mathscr{P}_B}|_{C_0}$ and $H^0(C,(\mathfrak{g}/\mathfrak{b})_{\mathscr{P}_B}|_{C})\longrightarrow (\mathfrak{g}/\mathfrak{b})_{\mathscr{P}_B}|_s$ factors through $H^0(C_0,(\mathfrak{g}/\mathfrak{b})_{\mathscr{P}_B}|_{C_0})$.
\end{proof}
\section{Kontsevich compactification of $\Bun_B^w$}\label{3}
In the last section, we have already constructed a Kontsevich compactification of $\Bun_B'$. An important feature of $\overline{\Bun}^{K,'}_{B,x}$ is that it admits a map to $B\backslash G/B$ and we can use it to construct a resolution of singularities of $\overline{\Bun}_B^w$, for any Weyl group element $w$. We start by introducing a stratification of $\Bun_B'$.
\subsection{Relative position}\label{rela}

$\Bun_B'$ admits a stratification indexed by the Weyl group. Note that $\Bun_B'$ admits a \textit{relative position} map
\begin{equation}
\begin{split}
      \textnormal{rp}: \Bun_B'&\to \pt/B\mathop{\times}\limits_{\pt/G}\pt/B\simeq  B\backslash G/B.
\end{split}
\end{equation}
 For any Weyl group element $w$, the fiber product of $\Bun_B'$ and the Bruhat cell $\Br^w\subset B\backslash G/B$ corresponding to $w$ over $B\backslash G/B$ is an algebraic stack, and we denote it by $\Bun_B^w$. Denote {by} $\overline{\Bun}_B^w$ the closure of $\Bun_B^w$ in $\overline{\Bun}_B'$.

 Consider the assignment which sends ($C, p, \mathscr{P}_G, {\epsilon_{C}}, \epsilon_{{x}}, s$) to $(\mathscr{P}_G, \epsilon_{{x}}, {\epsilon_{C}}|_s)$. It gives rise to a relative position map
 
 \begin{equation}\label{eq 3.2}
     \overline{\Bun}^{K,'}_{B,x}\longrightarrow \Bun_G'\underset{\pt/G}{\times}\pt/B\longrightarrow B\backslash G/B.
 \end{equation}
It is compatible with $\textnormal{rp}$ when restricted to $\Bun_B'$.

By Corollary \ref{cor 2.1}, the fiber product of $\overline{\Bun}^{K,'}_{B,x}$ and the Schubert {variety} $\Br^{\leq w}$ (i.e., closure of $\Br^{w}$) over $B\backslash G/B$ has a proper projection to {$\overline{\Bun}_B'$}, and the restriction of the projection is an isomorphism on ${\Bun}_B^w$.

However, in general, the fiber product of $\overline{\Bun}^{K,'}_{B,x}$ and the Schubert {variety} $\Br^{\leq w}$ (i.e., closure of $\Br^{w}$) over $B\backslash G/B$ is not a resolution of singularities of $\overline{\Bun}_B^w$. If we do not make any assumption on degree and defect, then {the author does not know} if the map from this fiber product to $\overline{\Bun}{}'_B$ factors through $\overline{\Bun}_B^w$. Also, the fiber product is not necessarily smooth, and the boundary of the isomorphism locus is not a normal crossing divisor. These barriers will lead to a difficulty of studying D-modules and $\ell$-adic sheaves on $\overline{\Bun}_B^w$.



The solution to solve these problems is to use the Bott–Samelson resolution (or, Hansen resolution, Demazure resolution).

Let us briefly recall its definition.
\subsection{Bott–Samelson resolution}
For $w\in W$, $w$ has a reduced expression $w= s_{i_1}s_{i_2}\cdots s_{i_l}$. We denote by $P_{i_k}$ the minimal parabolic subgroup of $G$ corresponding to $s_{i_k}$. The Bott–Samelson resolution $Z_w$ is defined to be {$P_{i_1}\overset{B}{\times} P_{i_2}\overset{B}{\times}  \cdots \overset{B}{\times}  P_{i_l}/ B$.} 

The multiplication gives rise to a map from $Z_w$ to the Schubert variety with respect to $w$. It provides a resolution of singularities of {the closure of} $BwB/B$ {in $G/B$}.

Furthermore, we notice that this map is left $B$-invariant. In particular, it induces a map of stacky quotients \[\tilde{Z}_w\longrightarrow {\Br}^{\leq w}.\] Here, $\tilde{Z}_w$ is the quotient of $Z_w$ by $B$.

A key feature of the Bott–Samelson resolution is that the boundary of the isomorphism locus ${\Br}^w$ is a normal crossing divisor.
\subsection{Construction of a resolution of singularities of $\overline{\Bun}_B^w$}
\begin{definition}\label{def 3.3.1}
For $w\in W$, we denote by $\BS^w$ the fiber product of $\overline{\Bun}^{K,'}_{B,x}$ and $\tilde{Z}_w$ over $B\backslash G/B$. 

Similarly, we define $\BS^{w, \lambda}_{\leq \mu}$ as the fiber product of $\overline{\Bun}^{K,',\lambda}_{B,x,\leq \mu}$ and $\tilde{Z}_w$ over $B\backslash G/B$. 
\end{definition}
We claim that $\BS^{w, \lambda}_{\leq \mu}$ provides a resolution of singularities of $\overline{\Bun}_{B,\leq \mu}^{w,\lambda}$ if the degree is very anti-dominant with respect to the defect.

\begin{prop}\label{imp}\
\begin{enumerate}[label={\upshape(\arabic*)}]
    \item $\BS^{w,\lambda}_{\leq \mu}$ is a smooth stack.
    \item The map $\pr$ induces a map \[\BS^{w,\lambda}_{\leq \mu}\longrightarrow \overline{\Bun}^{w,\lambda}_{B, \leq \mu}.\]Its restriction to $\Bun_{B}^{w,\lambda}$ is an isomorphism.
    \item The above map is proper. 
    \item The complement of $\Bun_{B}^{w,\lambda}$ in $\BS_{\leq \mu}^{w, \lambda}$ is a normal crossing divisor.
\end{enumerate}
\end{prop}
\begin{proof}
(1). By Proposition \ref{pro l}, $\overline{\Bun}^{K,',\lambda}_{B,x,\leq \mu}$ is smooth over $\mathscr{M}_{X,x}'\times \Bun_G''$. In particular, it is smooth over $\Bun_G''\simeq \Bun_G\underset{\pt/G}{\times} B\backslash G/B$. Note that $\Bun_G$ is smooth over the classifying stack of $G$, we obtain that $\overline{\Bun}^{K,',\lambda}_{B,x,\leq \mu}$ is smooth over $B\backslash G/B$.  By base change, $\BS^{w, \lambda}_{\leq \mu}$ is smooth over $\tilde{Z}_w$. Now, the smoothness of $\BS^{w, \lambda}_{\leq \mu}$ follows from the fact that $\tilde{Z}_w$ is smooth.

(2). By definition, the map $ \overline{\Bun}_{B,x}^{K,'}\longrightarrow \overline{\Bun}_B'$ of \eqref{17.12} preserves degree and defect, and is an isomorphism on $\Bun_B'$. So there is a map 
\begin{equation}\label{3.1 eq}
    \BS^{w,\lambda}_{\leq \mu}= \overline{\Bun}^{K,',\lambda}_{B,x,\leq \mu}\underset{B\backslash G/B}{\times}\tilde{Z}_w\longrightarrow \overline{\Bun}^{K,',\lambda}_{B,x,\leq \mu}\longrightarrow \overline{\Bun}^{',\lambda}_{B,\leq\mu}
\end{equation}
which is an isomorphism on $\Bun_{B}^{w,\lambda}$. Note that $\Bun_{B}^{w,\lambda}$ is open dense in $\BS^{w,\lambda}_{\leq \mu}$, so the above map factors through $\overline{\Bun}^{w,\lambda}_{B,\leq\mu}$.

(3). We only need to show that the map \eqref{3.1 eq} is proper. It follows from the properness of the map $\tilde{Z}_w\longrightarrow \Br^{\leq w}$ and Corollary \ref{cor 2.1}.

(4). To prove the complement of $\Bun_B^{w, \lambda}$ in $\BS^{w, \lambda}_{\leq \mu}$ is a normal crossing divisor, we need to notice that this complement is the preimage of a normal crossing divisor \[\partial (\mathscr{M}_{X,x}'\times  \tilde{Z}_w)=(\mathscr{M}_{X,x}\times \tilde{Z}_w)\setminus ((C=X, p=\id, s=x)\times \Br^w))\] under the map
\begin{equation}\label{ncd bs}
    \BS^{w,\lambda}_{\leq \mu}\longrightarrow \mathscr{M}_{X,x}'\times \tilde{Z}_w.
\end{equation}
The above map is the base change of 
\begin{equation}\label{3.3}
    \overline{\Bun}^{K,',\lambda}_{B,x,\leq \mu}\longrightarrow \mathscr{M}_{X,x}'\times \Bun_G''\longrightarrow \mathscr{M}_{X,x}'\times B\backslash G/B.
\end{equation}
Since the map \eqref{3.3} is smooth, \eqref{ncd bs} is smooth.
As a result, the complement of $\Bun_B^{w,\lambda}$ in $\BS^{w, \lambda}_{\leq \mu}$ is a normal crossing divisor as it is the preimage of a normal crossing divisor under a smooth map.
\end{proof}

\section{universally locally acyclicity of extension sheaves}\label{4}

As an application of the construction of $\BS^w$ in the last section, we will prove the universally locally acyclicity of the $!$-extension sheaf on $\overline{\Bun}_{B,\leq \mu}^{w,\lambda}$. The same claim holds for $*$-extension sheaf and the intersection cohomology sheaf with similar proofs. We will prove the universally locally acyclicity in the \textit{twisted} D-modules and $\ell$-adic sheaves setting. This generality is useful for the quantum Langlands program.

First of all, let us recall the meaning of the twisting here.
\subsection{Twisted setting}
Let $\mathscr{L}_{\Bun_T}$ be {the} {following} line bundle on $\Bun_T$, {its} fiber over a point $\mathscr{P}_T\in \Bun_T$ is given by $$\textnormal{det}R\Gamma(X, {O}_X)^{\dim (\mathfrak{t})}\otimes \bigotimes_{\check{\alpha}\in \check{\Delta}}\textnormal{det} R\Gamma(X, \check{\alpha}(\mathscr{P}_T)).$$
{In addition,} let $\mathscr{L}_{\Bun_G}$ be the determinant line bundle on $\Bun_G$, the fiber of $\mathscr{L}_{\Bun_G}$ over $\mathscr{P}_G$ is given by \[\textnormal{det}\ R\Gamma(X, \mathfrak{g}_{\mathscr{P}_G}),\] where $\mathfrak{g}_{\mathscr{P}_G}$ denotes the vector bundle associated with $\mathscr{P}_G$ and $\mathfrak{g}$.

$\overline{\Bun}_B$ admits natural projections to $\Bun_T$ and $\Bun_G$. 

\begin{center}
\begin{equation}
\begin{tikzpicture}[scale=1]

  \draw [->] (-1,1) -- (-2.5,0);
  \draw [->] (1,1) -- (2.5,0);
  \node at (0,1.2){$\overline{\Bun}_B$};
  \node [below] at (2.5,0){$\Bun_T$};
  \node [below] at (-2.5,0){$\Bun_G$};

\end{tikzpicture}
\end{equation}
\end{center}

We denote by $\mathscr{L}_{\overline{\Bun}_B}$ the ratio of the pullbacks of $\mathscr{L}_{\Bun_G}$ and $\mathscr{L}_{\Bun_T}$ to $\overline{\Bun}_B$. Let $\mathscr{L}_{\overline{\Bun}_B'}, \mathscr{L}_{\BS^w}$ be the pullback of $\mathscr{L}_{\overline{\Bun}_B}$ to the corresponding stacks.

Consider a multiplicative local system $c$ on $\mathbb{G}_m$, the category $\Shv^c(\overline{\Bun}_B')$ is defined as the category of sheaves on $\mathscr{L}_{\overline{\Bun}_{B}'}$ which is $\mathbb{G}_m$-equivariant with respect to the character sheaf $c$. Because the restriction of $\mathscr{L}_{\overline{\Bun}_B'}$ to ${\Bun}_B^w$ is canonically trivial for any $w\in W$, we have an equivalence of categories $\Shv^c(\Bun_{B}^w)\simeq \Shv(\Bun_B^w).$ In particular, we can consider the twisted sheaf corresponding to the constant sheaf on $\Bun_B^w$. Let us denote by $j^c_{!, \overline{\Bun}_B^w}\in \Shv^c(\overline{\Bun}_B')$ the $!$-extension of this twisted sheaf to $\overline{\Bun}_B'$ and by $j^c_{!,\overline{\Bun}^{w,\lambda}_{B,\leq \mu}}$ the restriction of $j^c_{!, \overline{\Bun}_B^w}$ to $\overline{\Bun}^{',\lambda}_{B,\leq \mu}$.
\subsection{Main theorem}\label{sect 4.2}
As an application of the construction of $\BS^w$, we want to prove the following theorem in the rest of this paper.

\begin{thm}\label{Kont}
If $\lambda\in \Lambda,\ \mu\in \Lambda^{pos}$ satisfy Condition \ref{X}, then $j^c_{!,\overline{\Bun}^{w,\lambda}_{B,\leq \mu}}$ is universally locally acyclic with respect to the projection $p_G'\colon \overline{\Bun}^{w, \lambda}_{B,\leq \mu}\longrightarrow \Bun_G'$.
\end{thm}

The following lemma of \cite[Proposition 4.3.1]{[Camp]} is obtained by using the description of the singular support given in \cite{[Sai]}. 

We assume that $f: Y\longrightarrow Z$ is a morphism between two smooth algebraic stacks locally of finite type.
\begin{lem}\label{ch}
Consider a normal crossing divisor $D=\sum r_i D_i$ in $Y$. We denote by $\Shv^c(Y)$ the category of D-modules or $\ell$-adic sheaves on $O(D)$ which are $\mathbb{G}_m$-equivariant with respect to the character sheaf $c$. Denote by $Y_n$ the locus of $D$ where locally can be regarded as the transversal intersection of  $n$ hyperplanes. If $Y_n\longrightarrow Z$ is smooth for all $n\geq 0$, then $j^c_{!,Y}$ is universally locally acyclic with respect to the map $$f:Y\longrightarrow Z$$
Here $j^c_{!,Y}\in \Shv^c(Y)$ is the $!$-extension of the twisted constant sheaf on $Y_0$ to $Y$.
\end{lem}

We denote by $j^c_{!, \BS^w}\in \Shv^c(\BS^w)$ the $!$-extension of the twisted constant sheaf on $\Bun_B^{w}$ to $\BS^w$ and by $j^c_{!, \BS^{w,\lambda}_{\leq \mu}}$ the restriction of $j^c_{!, \BS^w}$ to the open substack $\BS^{w, \lambda}_{\leq \mu}$. With the construction of $\BS^w$, we can prove Theorem \ref{Kont} easily.



\begin{proof} (of Theorem \ref{Kont})
Because $\mathscr{L}_{\BS^{w,\lambda}_{\leq \mu}}$ is canonically trivial on $\Bun_{B}^{w,\lambda}$, it is of the form ${O}(\sum_i r_i' D_i)$, where $D_i$ is an irreducible component of $\BS^{w, \lambda}_{\leq\mu}\setminus\Bun^{w,\lambda}_{B}$ for any $i$.

By Proposition \ref{imp} (4), $\BS^{w, \lambda}_{\leq\mu}\setminus \Bun^{w,\lambda}_{B}$ (i.e., the boundary of $\Bun^{w,\lambda}_{B}$ in $\BS^{w, \lambda}_{\leq \mu}$) is a normal crossing divisor. Furthermore, the composition of the follows maps 
\begin{equation}\label{long map}
    \begin{split}
        &\BS^{w, \lambda}_{\leq \mu}\simeq \overline{\Bun}^{K,',\lambda}_{B,x,\leq \mu}\underset{B\backslash G/B}{\times} \tilde{Z}_w\longrightarrow \mathscr{M}_{X,x}'\times \Bun_G''\underset{B\backslash G/B}{\times} \tilde{Z}_w\simeq\\
        &\simeq \mathscr{M}_{X,x}'\times (\Bun_G\underset{\pt/G}{\times} \pt/B\underset{\pt/G}{\times} \pt/B)\underset{B\backslash G/B}{\times} \tilde{Z}_w\simeq \mathscr{M}_{X,x}'\times \Bun_G\underset{\pt/G}{\times}\tilde{Z}_w\simeq\\ &\simeq \mathscr{M}_{X,x}'\times \Bun_G'\underset{\pt/B}{\times} \tilde{Z}_w
    \end{split}
\end{equation}
 is smooth. Since the locus $(\BS^{w, \lambda}_{\leq \mu})_n$ (notation of Lemma \ref{ch}) locally can be realized as the preimage of a smooth (relative to $\pt/B$) substack of $\mathscr{M}_{X,x}'\times \tilde{Z}_w$ along the composition of \eqref{long map} and $\mathscr{M}_{X,x}'\times \Bun_G'\underset{\pt/B}{\times} \tilde{Z}_w\longrightarrow \mathscr{M}_{X,x}'\times \tilde{Z}_w$, $(\BS^{w, \lambda}_{\leq \mu})_n$ is smooth over $\Bun_G'$ for any $n$. So, by Lemma \ref{ch}, $j^c_{!, \BS^{w, \lambda}_{\leq \mu}}$ is universally locally acyclic with respect to $\BS^{w,\lambda}_{\leq \mu}\longrightarrow \Bun_G'$.

We note that the restriction of the morphism $\BS^w\longrightarrow \overline{\Bun}_B'$ to $\Bun^{w, \lambda}_{B}$ is an isomorphism and the !-pushforward of $j^c_{!, \BS^{w, \lambda}_{\leq \mu}}$ is isomorphic to $j^c_{!, \overline{\Bun}^{w,\lambda}_{B,\leq \mu}}$. Furthermore, by Corollary \ref{imp}, $\BS^{w,\lambda}_{\leq \mu}\longrightarrow \overline{\Bun}^{w,\lambda}_{B,\leq \mu}$ is proper. Now, the theorem follows from the fact that proper pushforward preserves universally locally acyclicity.
\end{proof}

\subsection*{Acknowledgments}The author thanks Dennis Gaitsgory for his guidance, and for helpful discussions that simplify the original construction. The author thanks Justin Campbell for his detailed explanation of his proof in \cite{[Camp]} and many helpful discussions. The construction of the Kontsevich compactification $\BS^w$ is inspired by his construction in \cite{[Camp]}. {The author thanks the anonymous referee for pointing out errors in the former version and constructive comments, which make this paper more rigorous.} The author thanks Sergey Lysenko, Jonathan Wise, and Qile Chen for the helpful discussions contributed to the paper.

\end{document}